\let\accentvec\vec
\documentclass[runningheads,a4paper]{llncs}

\let\vec\accentvec
\usepackage{amssymb}

\usepackage{graphicx}
\usepackage{url}
\usepackage{algorithm}
\usepackage{algpseudocode}
\usepackage{multicol}
\usepackage{graphicx} 
\usepackage{amsmath}
\usepackage{amssymb}
\usepackage{mathrsfs}
\usepackage{setspace}
\usepackage{macro}
\usepackage{mathtools}
\usepackage{hyperref}

\usepackage{color}
\graphicspath{{pictures/}}
\begin{document}
\mainmatter  % start of an individual contribution

\title{ Certified Roundoff Error Bounds using Bernstein Expansions and Sparse Krivine-Stengle Representations}

\titlerunning{Roundoff errors with Bernstein expansions \& Krivine-Stengle representations}

\author{Alexandre Rocca\inst{1,2} \and Victor Magron \inst{1} \and Thao Dang\inst{1} }
%\thanks{Please note that the LNCS Editorial assumes that all authors have used
%the western naming convention, with given names preceding surnames. This determines
%the structure of the names in the running heads and the author index.}%
%
\authorrunning{A.Rocca, V.Magron, T.Dang}

\institute{VERIMAG/CNRS\\
700 avenue Centrale, 38400 Saint Martin D'H{\`e}res, France
\and 
UJF-Grenoble 1/CNRS, TIMC-IMAG, \\
UMR 5525, Grenoble, F-38041, France}

\maketitle

\begin{abstract}
Floating point error is an inevitable drawback of embedded systems implementation. Computing rigorous upper bounds of roundoff errors is absolutely necessary for the validation of critical software. This problem of computing rigorous upper bounds is even more challenging when addressing non-linear programs. In this paper, we propose and compare two new methods based on Bernstein expansions and sparse Krivine-Stengle representations, adapted from the field of the global optimization, to compute upper bounds of roundoff errors for programs implementing polynomial functions. We release two related software package FPBern and FPKiSten, and compare them with state of the art tools.

 We show that these two methods achieve competitive performance, while computing accurate upper bounds by comparison with other tools.
\vspace{-0.3cm}
\end{abstract}
%\vspace{-1cm}
\section{Introduction}
\label{introduction_intro}

Theoretical models, algorithms, and programs are often reasoned and designed in real algebra. However, their implementation on computers uses floating point algebra: this conversion from real numbers and their operations to floating point is not without errors. Indeed, due to finite memory and binary encoding in computers, real numbers cannot be exactly represented by floating point numbers. Moreover, numerous properties of the real algebra are not conserved such as commutativity or associativity.

The consequences of such imprecisions become particularly significant in safety-critical systems, especially in embedded systems which often include control components implemented as computer programs. When implementing an algorithm designed in real algebra, and initially tested on computers with single or double floating point precision, one would like to ensure that the roundoff error is not too large on more limited platforms (small processor, low memory capacity) by computing their accurate upper bounds.

% In this case, it becomes useful to compute an upper bound of the roundoff error caused by the floating point implementation.\\
For programs implementing linear functions, SAT/SMT solvers, as well as affine arithmetic, are efficient tools to obtain good upper bounds. When extending to programs with non-linear polynomial functions, the problem of determining a precise upper bound becomes substantially more difficult, since polynomial optimization problems are in general NP-hard~\cite{laurent2009sums}. We can cite at least three closely related and recent frameworks designed to provide upper bounds of roundoff errors for non-linear programs. \texttt{FPTaylor} \cite{fptaylor} is a tool based on Taylor-interval methods, while \texttt{Rosa} \cite{rosa} combines SMT with interval arithmetic.  \texttt{Real2Float} \cite{real2float} relies on Putinar representations of positive polynomials while exploiting sparsity in a similar way as the second method that we propose in this paper. %sums of squares relaxations and is, as we will see, the closest to the second method we propose in this paper. %, which is based on sparse Krivine-Stengle representations.

We introduce two methods, coming from the field of polynomial optimization, to compute upper bounds on roundoff errors of polynomial programs.  The first method is based on Bernstein expansions of polynomials, while the second relies on sparse Krivine-Stengle certificates for positive polynomials. 
In practice, these methods (presented in Section~\ref{contributions}) provide accurate bounds at a reasonable computational cost. 
Indeed, the size of the Bernstein expansion used in the first method as well as the size of the LP relaxation problems considered in the second method are both linear w.r.t.~the number of roundoff error variables. 
%In the sequel, we use a simple example to illustrate our two methods.
% and a first glance to the contribution, while in Section \ref{related_works} we will compare our methods to the current state of the art. Finally, in Section \ref{key_contributions} we will summarize the key contributions of our paper.
\subsection{Overview}
\label{overview}
Before explaining in detail each method, let us first illustrate the addressed problem on an example. Let $f$ be the degree two polynomial defined by:
\[ 
f(x) := x^2 - x~,\quad \forall x\in {X} = [0,1].
\]
When approximating the value of $f$ at a given real number $x$, one actually computes the floating point result $\hat{f} = \hat{x}\otimes\hat{x} \ominus \hat{x}$, with all the real operators $+$,$-$,$\times$ being substituted by their associated floating point operators $\oplus$, $\ominus$, $\otimes$, and $x$ being represented by the floating point number $\hat{x}$ (see Section \ref{preliminaries_floating_point} for more details on floating point arithmetics). A first simple rounding model consists of introducing an error term $e_i$ for each floating point operation, as well as for each floating point variable. For instance, $\hat{x}\otimes\hat{x}$ corresponds to $((1+e_1) \, x \,  (1+e_1) \, x) \, (1+e_2)$, where $e_1$ is the error term between $x$ and $\hat{x}$, and $e_2$ is the one associated to the operation $\otimes$.
Let $\eb$ be the vector of all error terms $e_i$. Given $e_i \in [ - \varepsilon ,\varepsilon ]$ for all $i$, with $ \varepsilon $ being the machine precision, we can write the floating point approximation $\hat{f}$ of $f$ as follows:
\[ 
\hat{f}(x,\eb) = (((1+e_1)x  (1+e_1) x)  (1+e_2) - x (1+e_1)) (1+e_3).
\]
Then, the absolute roundoff error is defined by:
\[
r(x,\eb) := \max\limits_{\begin{subarray}{1}x\in [0,1]\\ \eb\in[ - \varepsilon ,\varepsilon ]^3\end{subarray}}(|\hat{f}(x,\eb) - f(x)|)~~.
\]
% for all $(x,\eb)\in [0,1]\times[ - \varepsilon ,\varepsilon ]^3$.
However, we can make this computation easier with a slight approximation: $|\hat{f}(x,\eb)-f(x)| \leq |l(x,\eb)| + |h(x,\eb)|$ with $l(x,\eb)$ being the sum of the terms of $(\hat{f}(x,\eb)-f(x))$ which are linear in $\eb$, and $h(x,\eb)$ the sum of the terms which are non-linear in $\eb$. The term $|h(x,\eb)|$ can then be over-approximated by $O(|\eb|^2)$ which is \emph{in general} negligible compared to $|l(x,\eb)|$, and can be bounded using standard interval arithmetic. For this reason, we focus on computing an upper bound of $|l(x,\eb)|$. In the context of our example, $l(x,\eb)$ is given by:
\begin{align}
\label{eq:l}
l(x,\eb) = (2x^2-x) e_1 + x^2 e_2 + (x^2-x) e_3.
\end{align} 
%From now on, and until the end of the Section \ref{overview}, we take $\eb\in [-1,1]^3$, instead of $[-\varepsilon,\varepsilon]^3$: this way, the resulting roundoff error is proportional to $\varepsilon$.\\
We divide each error term $e_j$ by $\varepsilon$, and then consider the (scaled) linear part $l':=\frac{l}{\varepsilon}$ of the roundoff error with the error terms $\eb\in [-1,1]^3$.
For all $x\in [0,1]$, and $\eb\in [-1,1]^3$, one can easily compute a valid upper bound of $|l'(x,\eb)|$ with interval arithmetic. Surcharging the notation for elementary operations $+,-,\times$ in interval arithmetic, one has $l'(x,\eb) \in ([-0.125,1]\times[-1,1]+[0,1]\times[-1,1]+[-0.25,0]\times[-1,1])=[-2.25,2.25]$, yielding $|l(x,\eb)|\leq 2.25\varepsilon$.\\
Using the first method based on Bernstein expansion detailed in Section \ref{contributions_bernstein}, we compute $2\varepsilon$ as an upper bound of $|l(x,\eb)|$ in $0.23$s using \texttt{FPBern(b)} a rational arithmetic implementation. With the second method based on sparse Krivine-Stengle representation detailed in Section \ref{contributions_handelman}, we also compute an upper bound of $2\varepsilon$ in $0.03$s.\\
 Although on this particular example, the method based the sparse Krivine-Stengle representation appears to be more time-efficient,  in general the computational cost of the method based on Bernstein  expansions is lower. For this example, the bounds provided by both methods are tighter than the ones determined by interval arithmetic.
We emphasize the fact that the bounds provided by our two methods can be certified. Indeed, in the first case, the Bernstein coefficients (see Sections \ref{preliminaries_bernstein} and \ref{contributions_bernstein}) can be computed either with rational arithmetic or certified interval arithmetic to ensure guaranteed values of upper bounds. In the second case, the positivity certificates are directly provided by sparse Krivine-Stengle representations.

\subsection{Related Works}
\label{related_works}
We first mention two tools, based on positivity certificates, to compute roundoff error bounds. The first tool, related to~\cite{Constantinides}, relies on a similar approach to our second method. It uses dense Krivine-Stengle representations of positive polynomials to cast the initial problem as a finite dimensional LP problem. To reduce the size of this possibly large LP, \cite{Constantinides} provides heuristics to eliminate some variables and constraints involved in the dense representation. However, this approach has the main drawback of loosing the property of convergence toward optimal solutions of the initial problem.
Our second method uses sparse representations and is based on the previous works by~\cite{schweighofer06} and~\cite{sbsos}, allowing to ensure the convergence towards optimal solutions while greatly reducing the computational cost of LP problems. Another tool, \texttt{Real2Float} \cite{real2float}, exploits sparsity in the same way while using Putinar representations of positive polynomials, leading to solving semidefinite (SDP) problems.  Bounds provided by such SDP relaxations are in general more precise than LP relaxations~\cite{Lasserre2009Moments}, but the solving cost is higher.\newline
Several other tools are available to compute floating point roundoff errors. 
SMT solvers are efficient when handling linear programs, but often provide coarse bounds for non-linear programs, e.g.~when the analysis is done in isolation~\cite{rosa}. The \texttt{Rosa}~\cite{rosa} tool is a solver mixing SMT and interval arithmetic which compiles functional \texttt{SCALA} programs implementing non-linear functions (involving $/,\surd$, and polynomials) as well as conditional statements. SMT solvers are theoretically able to output certificates which can be validated externally afterwards. 
%The \texttt{FPTaylor} \cite{fptaylor} tool, relies on symbolic Taylor interval method, and branch and bound techniques, to provide the upper error bounds of the studied problems.
\texttt{FPTaylor} tool \cite{fptaylor}, relies on \emph{Symbolic Taylor expansion} method, which consists of a branch and bound algorithm based on interval arithmetic. %$s_j$, $j=1,\dots,m$ when $l(\xb,\eb)=\sum_{j=1}^m s_j(\xb)e_j$, and finally obtain an upper bound of $|l|$. 
Bernstein expansions have been extensively used to handle systems of polynomial equations \cite{mourrain2009,smithThesis} as well as systems of polynomial inequalities (including polynomial optimization), see for example \cite{smithThesis,dreossiHSCC,munoz13}. Yet, to the best of our knowledge, there is no tool based on Bernstein expansions in the context of roundoff error computation.
The \texttt{Gappa} tool provides certified bounds with elaborated interval arithmetic procedure relying on multiple-precision dyadic fractions. % to provide fast, and yet sometimes coarser error bounds. 
The static analysis tool \texttt{FLUCTUAT} \cite{fluctuat}  performs forward computation (by contrast with optimization) to analyze floating point \texttt{C} programs. Both \texttt{FLUCTUAT} and \texttt{Gappa} use a different rounding model (see Section \ref{preliminaries_floating_point}), also available in \texttt{FPTaylor}, that we do not handle in our current implementation.
Some tools also allow formal validation of certified bounds. \texttt{FPTaylor}, \texttt{Real2Float} \cite{real2float}, as well as \texttt{Gappa} \cite{gappa} provide formal proof certificates, with \texttt{HOL-Light} \cite{hollight} for the first case, and \texttt{Coq} \cite{CoqProofAssistant} for the two other ones. 
% We will not compare with it in the Section \ref{implementation}, since \texttt{Gappa} is well adapted to the improved rounding model, while we focus on the simple rounding model.\\

\subsection{Key Contributions}
\label{key_contributions}
Here is a summary of our key contributions:
\begin{itemize}
\item[$\blacktriangleright$] We present two new methods to compute upper bounds of floating point roundoff errors for programs implementing multivariate polynomial functions with input variables constrained to boxes. The first one is based on Bernstein expansions and the second one relies on sparse Krivine-Stengle representations.
We also propose a theoretical framework to guarantee the validity of upper bounds computed with both methods (see Section~\ref{contributions}). In addition, we give an alternative shorter proof in Section \ref{preliminaries_handelman} for the existence of Krivine-Stengle representations for sparse positive polynomials (proof of Theorem~\ref{spkrivrep_th}). %using \cite[Lemma 3]{schweighofer06}, instead of Putinar's Theorem as in \cite{sbsos}.
%
%Both of them have the advantage of being linear in the number of error terms.
% While Bernstein is the fastest method (in the case of the C++ implementation without rational arithmetic), it is limited to box domain (or equivalent to a box by a linear transformation). Krivine-Stengle method is slightly slower, but it carries the advantage of being applicable on more general domain, similarly to the SDP relaxations exposed in \cite{real2float}. 
\item[$\blacktriangleright$] We release two software packages based on each method. The first one, called \texttt{FPBern}\footnote{ \url{https://github.com/roccaa/FPBern} }, computes the bounds using the Bernstein expansions, with two modules built on top of the software related to \cite{dreossiHSCC}: \texttt{FPBern(a)} is a \texttt{C++} module using double precision floating point arithmetic while \texttt{FPBern(b)} is a \texttt{Matlab} module using rational arithmetic. % through \texttt{Matlab} \texttt{symbolic toolbox}. 
%These two modules are . 
The second one \texttt{FPKriSten}\footnote{ \url{https://github.com/roccaa/FPKriSten} } computes the bounds using Krivine-Stengle representations in \texttt{Matlab}. \texttt{FPKriSten} is built on top of the implementation related to \cite{sbsos}.
\item[$\blacktriangleright$] We compare our two methods implemented in \texttt{FPBern} and \texttt{FPKriSten} to three state-of-the-art methods. Our new methods have similar precision with the compared tools (\texttt{Real2Float, Rosa, FPTaylor}). At the same time, \texttt{FPBern(a)} shows an important time performance improvement, while \texttt{FPBern(b)} and \texttt{FPKriSten} has similar time performances compared with the other tools, yielding promising results. 
%Even if the current implementation is limited on the type of programs it can handle, we can still provide a proof of concept for both of the proposed methods. 
\end{itemize}
The rest of the paper is organized as follows: in Section \ref{preliminaries}, we give basic background on floating point arithmetic, the Bernstein expansions and Krivine-Stengle representations. In Section \ref{contributions} we give the main contributions, that is the computation of roundoff error bounds using Bernstein expansions and sparse Krivine-Stengle representations. Finally, in Section \ref{implementation} we compare the performance and precision of our two methods with the existing tools, and show the advantages of our tools.
%In Section \ref{contributions} we will give the details on the methods to efficiently compute the roundoff error using the Bernstein expansion and the Krivine-Stengle representation. Finally, in the Section \ref{implementation} we will compare our results to three recent tools: \texttt{FPTaylor}, \texttt{Rosa}, and \texttt{Real2Float}.\\
\section{Preliminaries}
\label{preliminaries}
We first recall useful notation on multivariate calculus. %, used hold for Sections \ref{preliminaries} and \ref{contributions}.\\
For $\mathbf{x} = (x_1,\ldots, x_n)\in \mathbb{R}^n$ and the multi-index $\mathbf{\ab} = (\alpha_1,\ldots, \alpha_n)\in \mathbb{N}^n$, we denote by
%\begin{itemize}
 $\xb^{\ab}$ the product $\prod_{i = 1}^n x_i^{\alpha_i}$. We also define $|\ab| = |\alpha_1| + \ldots + |\alpha_n|$, $\bzero=(0,\ldots,0)$ and $\bone = (1,\ldots,1)$.\newline
The notation $\sum_{\ab}$ is the nested sum $\sum_{\alpha_1}\ldots\sum_{\alpha_n}$. Equivalently we have $\prod_{\ab}$ which is equal to the nested product $\prod_{\alpha_1}\ldots\prod_{\alpha_n}$. \newline
Given another multi-index $\mathbf{d} = (d_1,\ldots, d_n)\in \mathbb{N}^{n}$, the inequality $\ab < \mathbf{d}$ (resp.~$\ab \leq \mathbf{d}$) means that the inequality holds for each sub-index: $\alpha_1 < d_1, \ldots, \alpha_n < d_n$ (resp.~$\alpha_1 \leq d_1, \ldots, \alpha_n \leq d_n$). Moreover, the binomial coefficient $\binom{\mathbf{d}}{\ab}$ is the product $\prod_{i = 1}^n \binom{d_i}{\alpha_i}$.\newline
Let $\mathbb{R}[\xb]$ be the vector space of multivariate polynomials.  Given $f\in\mathbb{R}[\xb]$,  we associate a {\em multi-degree} $\mathbf{d} = (d_1, \dots, d_n)$ to $f$, with each $d_i$ standing for the degree of $f$ with respect to the variable $x_i$. 
Then, we can write $f(\mathbf{x}) = \sum_{\gb\leq\mathbf{d}}a_{\gb} \xb^{\gb}$, with $a_{\gb}$ (also noted $(f)_{\gb}$) being the coefficients of $f$ in the monomial basis and each $\gb\in \mathbb{N}^n$ is a multi-index. The degree $d$ of $f$ is given by $d := \max_{\{\gb : a_{\gb}\neq 0\}} | \gb |$. As an example, if $f(x_1,x_2) = x_1^4x_2 +x_1^1x_2^3$ then $\mathbf{d} = (4,3)$ and $d = 5$. For the polynomial $l$ used in Section~\ref{overview}, one has $\mathbf{d} = (2,1,1,1)$ and $d = 3$.
%\end{itemize}
%
\subsection{Floating Point arithmetic}
\label{preliminaries_floating_point}
This section gives background on floating point arithmetic, inspired from material available in~\cite[Section 3]{fptaylor}. The IEEE754 standard \cite{IEEE754} defines a binary floating point number as a triple significant, sign, and exponent (denoted by $sig,sgn,exp$) which represents the numerical value of $(-1)^{sgn}\times sig \times 2^{exp}$. The standard describes 3 formats (32, 64, and 128 bits) which vary by the size of the significant and the exponent, as well as special values (such as NaN, the infinities).
% but we will not go into details about them.
Denoting by $\mathbb{F}$ the set of floating point numbers, we call rounding operator, the function $\textnormal{rnd}:\mathbb{R}\rightarrow\mathbb{F}$ which takes a real number and returns the closest floating point number rounded to the nearest, toward zero, or toward $\pm\infty$. A simple model of rounding is given by the following formula:
\[ \textnormal{rnd}(x) = x(1+e) + u,\]
with $|e|\leq\varepsilon$, $|u|\leq \mu$ and $eu=0$. The value $\varepsilon$ is the maximal relative error (given by the machine precision \cite{IEEE754}), and $\mu$ is the maximal absolute error for numbers very close to $0$. For example, in the single (32 bits) format, $\varepsilon$ is equal to $2^{-24}$ while $\mu$ equals $2^{-150}$. It is clear that in general $\mu$ is negligible compared to $\varepsilon$, thus we neglect terms depending on $u$ in the remainder of this paper.\newline
Given an operation op : $\mathbb{R}^n\rightarrow\mathbb{R}$, let $\textnormal{op}_{\textnormal{FP}}$ be the corresponding floating point operation. An operation is exactly rounded when
$\textnormal{op}_{\textnormal{FP}}(\xb) = \textnormal{rdn}(\textnormal{op}(\xb))$, for all $\xb \in \mathbb{R}^n$. \newline
In the IEEE754 standard the following operations are defined as exactly rounded: $+,-,\times,/,\surd,$ and the \texttt{fma} operation\footnote{The \texttt{fma} operator is defined by \texttt{fma}($x,y,z$)=$x \times y+z$.}. It follows that for those operations we have the continuation of the simple rounding model  $\textnormal{op}_{\textnormal{FP}}(\xb) =\textnormal{op}(\xb)(1+e)$. \newline
The previous rounding model is called ``simple" in contrast with more improved rounding model. Given the function $\textnormal{pc}(x) = \max_{k\in\mathbb{Z}}\{2^k : 2^k < x\}$, then the improved rounding model is defined by:
$
\textnormal{op}_{\textnormal{FP}}(\xb) =\textnormal{op}(\xb) +  \textnormal{pc}(\textnormal{op}(\xb))
$, for all $\xb \in \mathbb{R}^n$.
As the function pc is piecewise constant, this rounding model needs design of algorithms based on successive subdivisions, which is not currently handled in our methods. Combining branch and bound algorithms with interval arithmetic is adapted to roundoff error computation with such rounding model, which is the case with \texttt{FLUCTUAT}\cite{fluctuat}, \texttt{Gappa}\cite{gappa}, and \texttt{FPTaylor} \cite{fptaylor}. 
\subsection{Bernstein Expansion of Polynomials}
\label{preliminaries_bernstein}
In this section we give mandatory background on the Bernstein expansion for the contribution detailed in Section \ref{contributions_bernstein}.
Given a multivariate polynomial $f \in \mathbb{R}[\xb]$, we recall how to compute a lower bound of $\underline{f}^* :=\min_{\xb\in[0,1]^n}f(\xb)$.
The next result can be retrieved in~\cite[Theorem 2]{berngarloff}:
\begin{theorem}[Multivariate Bernstein expansion]
Given a multivariate polynomial $f$ and a degree $\kb\geq \db$ with $\db$ the multi-degree of $f$, then the Bernstein expansion of multi-degree $\kb$ of $f$ is given by:\\
\begin{equation}
f(\mathbf{x}) = \sum\limits_{\gb}a_{\gb} \xb^{\gb} = \sum\limits_{\ab\leq\kb} b_{\ab}^{(f)}\mathbf{B}_{\kb,\ab}(\mathbf{x}).
\end{equation}
where $b_{\ab}^{(f)}$ (also denoted by $b_{\ab}$ when there is no confusion) are the Bernstein coefficients (of multi-degree $\kb$) of $f$, and $\Bb_{\kb,\ab}(\xb)$ are the Bernstein basis polynomials defined by $\Bb_{\kb,\ab}(\xb) := \prod_{i = 1}^n B_{k_i,\alpha_i}(x_i)$ and $B_{k_i,\alpha_i}(x_i) := \binom{k_i} {\alpha_i} x_i^{\alpha_i}(1-x_i)^{k_i - \alpha_i}$. The Bernstein coefficients are given by the following formulas:
\begin{equation}
b_{\ab} = \sum_{\bb < \ab}\frac{ \binom{\ab} {\bb} }{ \binom{\kb} {\bb} } a_{\bb}, \quad \bzero\leq \ab \leq \kb.
\end{equation}
\end{theorem}
The Bernstein expansion having numerous properties, we give only four of them which are useful for Section \ref{contributions_bernstein}. For a more exhaustive introduction to Bernstein expansion, as well as some proof of the basic properties, we refer the interested reader to \cite{smithThesis}.
\begin{property}[{Cardinality \cite[(3.14)]{smithThesis}}]
\label{card_prop}
The number of Bernstein coefficients in the Bernstein expansion (of multi-degree $\kb$) is equal to $ (\kb+\bone)^{\bone} = \prod_{i = 1 }^n(k_i+1)~~.$
\end{property}
\begin{property}[{Linearity \cite[(3.2.3)]{smithThesis}}]
\label{lin_prop}
Given two polynomials $p_1$ and $p_2$, one has:%the Bernstein coefficients of $c p_1+p_2$ is the sum of the corresponding Bernstein coefficients of $p_1$ and $p_2$:
$$ b_{\ab}^{(c p_1+p_2)} = c b_{\ab}^{(p_1)} + b_{\ab}^{(p_2)} ~, \quad \forall c \in\mathbb{R},$$
where Bernstein expansions with same multi-degrees are considered.
\end{property}
\begin{property}[{Enclosure \cite[(3.2.4)]{smithThesis}}]
\label{enclosure_prop}
The minimum (resp.~maximum) of a polynomial $f$ over $[0, 1]^n$ can be lower bounded (resp.~upper bounded) by the minimum (resp.~maximum) of its Bernstein coefficients:
$$ \min_{\ab \leq \kb} b_{\ab} \leq f(\xb) \leq \max_{\ab \leq \kb} b_{\ab},~~ \forall \xb\in[0,1]^n ~~.$$ 
\end{property}
\begin{property}[{Sharpness \cite[(3.2.5)]{smithThesis}}]
\label{sharp_prop}
If the minimum (resp.~maximum) of the $b_{\ab}$ is reached for $\ab$ in a corner of the box $[0,k_1]\times\dots\times[0,k_n]$, then $b_{\ab}$ is the minimum (resp.~maximum) of $f$ over $[0,1]^n$.
\end{property}

Property \ref{card_prop} gives the maximal computational cost needed to find a lower bound of $f^*$ for a Bernstein expansion of fixed multi-degree $\kb$. Property \ref{enclosure_prop} is used to bound from below optimal values, while Property \ref{sharp_prop} allows to determine if the lower bound is optimal. 
\subsection{Dense and Sparse Krivine-Stengle Representations}
\label{preliminaries_handelman}
In this section, we first give the necessary background on Krivine-Stengle representations, used in the context of polynomial optimization. Then, we present a sparse version based on \cite{schweighofer06}. These notions are applied later in Section \ref{contributions_handelman}.
\paragraph{\textbf{Dense Krivine-Stengle representations.}}
Krivine-Stengle certificates for positive polynomials can first be found in \cite{krivineanneaux,stengle} (see also \cite[Theorem 1(b)]{bsos}). Such certificates give representations of positive polynomials over a set $\cSet = \{\xb\in\mathbb{R}^n : 0\leq g_i(\xb)\leq 1,~i=1,\dots,p\}$, with $g_1, \dots, g_p \in\mathbb{R}[\xb]$. The compact set $\cSet$ is a basic semialgebraic set, since it is defined as a conjunction of polynomial inequalities. \\
Given $\ab = (\alpha_1,\dots,\alpha_p)$ and $\bb=(\beta_1,\dots,\beta_p)$, let us define the polynomial $h_{\ab,\bb}(\xb) = \mathbf{g}^{\ab} (\mathbf{1} - \mathbf{g})^{\bb} = \prod_{i=1}^p g_i^{\alpha_i}(1-g_i)^{\beta_i}$. \newline
For instance on the two-dimensional unit box, one has $n=p=2$, $\cSet = [0,1]^2 = \{\xb\in\mathbb{R}^2 : 0\leq x_1\leq 1 \,, \ 0\leq x_2\leq 1\}$. With $\ab = (2,1)$ and $\bb = (1,3)$, one has $h_{\ab,\bb}(\xb) = x_1^2 x_2 (1 - x_1) (1 - x_2)^3$.
\begin{theorem}[Dense Krivine-Stengle representations]
\label{dense_stengle_th}
Let $\pfunc\in\mathbb{R}[\xb]$ be a positive polynomial over $\cSet$. Then there exist $k\in\mathbb{N}$ and a finite number of nonnegative weights $\lambda_{\ab,\bb}\geq 0$ such that:
	\begin{equation}
	\label{eq:ks}
		\pfunc(\xb) = \sum_{|\ab+\bb|\leq k}\lambda_{\ab,\bb}h_{\ab,\bb}(\xb), \quad \forall\xb \in \mathbb{R}^n.
	\end{equation}
\end{theorem}	
It is possible to compute the weights $\lambda_{\ab,\bb}$ by identifying in the monomial basis the coefficients of the polynomials in the left and right sides of~\eqref{eq:ks}. Denoting by $(\pfunc)_{\gb}$ the monomial coefficients of $\pfunc$, with $\gb\in\mathbb{N}_{k}^n := \{\gb\in\mathbb{N}^n : |\gb|\leq k\}$, the $\lambda_{\ab,\bb}$ fulfill the following equalities:
\begin{equation}
\pfunc_{\gb} = \sum_{|\ab+\bb|\leq k}\lambda_{\ab,\bb}(h_{\ab,\bb})_{\gb}, \quad \forall \gb\in\mathbb{N}_k^n.
\end{equation}
\paragraph{\textbf{Global optimization using the dense Krivine-Stengle representations.}}

Here we consider the polynomial minimization problem  $\underline{f}^* : = \min_{\xb\in \cSet }f(\xb)$, with $f$ a polynomial of degree $d$. We can rewrite this problem as the following infinite dimensional problem:
\begin{equation}
	\label{infopti_eq}
     \begin{aligned}
        \underline{f}^* :=& \max\limits_{t\in\mathbb{R}}~~t,\\
        	  & \text{s.t.}~~f(\xb) - t \geq 0 \,, \quad \forall \xb\in  \cSet .  \\   
      \end{aligned}
\end{equation}
The idea is to look for a hierarchy of finite dimensional linear programming (LP) relaxations by using Krivine-Stengle representations of the positive polynomial $\pfunc = f - t$ involved in Problem~\eqref{infopti_eq}. Applying Theorem \ref{dense_stengle_th} to this polynomial, we obtain the following LP problem for each $k \geq d$:
	\begin{equation}
	\label{eq:pk}
	     \begin{aligned}
        	  p^*_k := \max\limits_{t,\lambda_{\ab,\bb}} 
        	  \quad & t,\\
        	  \text{s.t } \quad & (f - t)_{\gb} = \sum_{|\ab+\bb|\leq k}\lambda_{\ab,\bb}(h_{\ab,\bb})_{\gb} \,, \quad \forall \gb\in\mathbb{N}_k^n \,,\\   
      		  \quad & \lambda_{\ab,\bb}\geq 0.
      	   \end{aligned}
	\end{equation}
As in~\cite[(4)]{bsos}, one has:
\begin{theorem}[Dense Krivine-Stengle LP relaxations]
\label{denseLPrelax_th}
The sequence of optimal values $(p^*_k)$ satisfies  $p^*_k\rightarrow \underline{f}^*$ as $k\rightarrow +\infty$. Moreover each $p^*_k$ is a lower bound of $\underline{f}^*$.
\end{theorem} 
At fixed $k$, the total number of variables of Problem~\eqref{eq:pk} is given by the number of $\lambda_{\ab,\bb}$ and $t$, that is $\binom{2n+k}{k} +1$. The number of constraints is equal to the cardinality of $\mathbb{N}_{k}^n$, which is $\binom{n+k}{k}$.
\paragraph{\textbf{Sparse Krivine-Stengle representations.}}
We now explain how to derive less computationally expensive LP relaxations, by relying on sparse Krivine-Stengle representations. 
 For $I \subseteq \{1,\dots,n\}$, let $\mathbb{R}[\xb, I]$ be the ring of polynomials restricted to the variables $\{x_i~:~i\in I \}$. 
% The ring $\mathbb{R}[\xb, I ]$ is a subring of $\mathbb{R}[\xb]$ and $g\in\mathbb{R}[\xb, I]$ implies $g:\mathbb{R}^{\# I}\rightarrow\mathbb{R}$ as well as $g:\mathbb{R}^{n}\rightarrow\mathbb{R}$.\\
We borrow the notion of a sparsity pattern from \cite[Assumption 1]{sbsos}:
\begin{definition}[Sparsity Pattern]
\label{sparse_def}
Given $m\in\mathbb{N}$, $I_j\subseteq\{1,\dots,n\}$, and $J_j \subseteq\{1,\dots,p\}$ for all $j = 1,\dots,m$, a sparsity pattern is defined by the four following conditions:
%% Note for redaction later:
% n-> number of variable
% p-> number of constraint gi
% m-> number of sparse blocks
% r-> iteration index on the sparse blocks
\begin{itemize}
\item $f$ can be written as: $f = \sum_{j=1}^{m}f_j$ with $f_j\in\mathbb{R}[\xb,I_j]$, 
\item $g_i \in\mathbb{R}[\xb,I_j]$ for all $i \in J_j$, for all $j = 1, \dots, m$,
\item $\bigcup_{j=1}^m I_j = \{1,\dots,n\}$ and $\bigcup_{j=1}^m J_j = \{1,\dots,p\}$,
\item (Running Intersection Property)~for all $j=1,\dots,m-1$, there exists $s\leq j$ s.t. $I_{j+1} \cap \bigcup_{i=1}^j I_i \subseteq I_s$.
\end{itemize}
\end{definition}
As an example, the four conditions stated in Definition~\ref{sparse_def} are satisfied while considering $f(\xb) = x_1 x_2 + x_1^2 x_3$ on the hypercube $\cSet = [0, 1]^3$. Indeed, one has $f_1(\xb) = x_1 x_2 \in \mathbb{R}[\xb,I_1]$, $f_2(\xb) = x_1^2 x_3 \in \mathbb{R}[\xb,I_2]$ with $I_1 = \{1,2\}$, $I_2 = \{1,3\}$. Taking $J_1 = I_1$ and $J_2 = I_2$, one has $g_i = x_i \in \mathbb{R}[\xb,I_j]$ for all $i \in I_j$, $j=1,2$.

Let us consider a given sparsity pattern as stated above. By noting $n_j = |I_j|$, $p_j = |J_j|$, then the set $ \cSet  = \{\xb\in\mathbb{R}^n : 0\leq g_i(\xb)\leq 1, \, i=1,\dots,p\}$ yields subsets $ \cSet_j = \{\xb\in\mathbb{R}^{n_j} : 0\leq g_i(\xb)\leq 1,~i\in J_j\}$, with $j = 1,\dots,m$. If $ \cSet $ is a compact subset of $\mathbb{R}^n$ then each $\cSet_j$ is a compact subset of $\mathbb{R}^{n_j}$.
As in the dense case, let us note $h_{\ab_j,\bb_j} :=  \mathbf{g}^{\ab_j}(\mathbf{1}- \mathbf{g})^{\bb_j}$, for given $\ab_j, \bb_j \in \mathbb{N}^{n_j}$. \newline

The following result, a sparse variant of Theorem~\ref{dense_stengle_th}, can be retrieved from~\cite[Theorem 1]{sbsos} but we also provide here a shorter alternative proof by using~\cite{schweighofer06}.
\begin{theorem}[Sparse Krivine-Stengle representations]
\label{spkrivrep_th}
Let $f,g_1,\dots,g_p\in\mathbb{R}[\xb]$ be given and assume that there exist $I_j$ and $J_j$, $j=1,\dots,m$, which satisfy the four conditions stated in Definition~\ref{sparse_def}. If $f$ is positive over $\mathbf{K}$, then there exist $\phi_j \in \mathbb{R}[\xb,I_j]$, $j=1,\dots,m$ such that $f=\sum_{j=1}^m\phi_j$ and $\phi_j > 0$ over $\cSet_j$. In addition, there exist $k\in\mathbb{N}$ and finitely many nonnegative weights $\lambda_{\ab_j,\bb_j}$, $j=1,\dots,m$, such that:
\begin{equation}
\phi_j = \sum_{|\ab_j+\bb_j|\leq k}\lambda_{\ab_j,\bb_j} h_{\ab_j,\bb_j}~, \quad j=1,\dots,m.
\end{equation}
\end{theorem}
\begin{proof}
From \cite[Lemma 3]{schweighofer06}, there exist $\phi_j \in \mathbb{R}[\xb,I_j]$ such that $f=\sum_{j=1}^m\phi_j$ and $\phi_j>0$ on $\mathbf{K}_j$. 
Applying Theorem~\ref{dense_stengle_th} on each $\phi_j$, there exist $k_j\in\mathbb{N}$ and finitely many nonnegative weights $\lambda_{\ab_j,\bb_j}$ such that $\phi_j = \sum_{|\ab_j+\bb_j|\leq k_j}\lambda_{\ab_j,\bb_j}h_{\ab_j,\bb_j}.$
With $k = \max_{1 \leq j \leq m} \{k_j\}$, we complete the representations with as many zero $\lambda$ as necessary, we obtain the desired result.\hfill\qed
\end{proof}
%
%\begin{remark}
%
In Theorem~\ref{spkrivrep_th}, one assumes that $f$ can be written as the sum $f = \sum_{j=1}^m f_j$, where each $f_j$ is not necessarily positive. The first result of the theorem states that that $f$ can be written as another sum $f = \sum_{j=1}^m \phi^j$, where each $\phi_j$ is now positive.
%\end{remark}
%
As in the dense case, the $\lambda_{\ab_j,\bb_j}$ can be computed by equalizing the coefficients in the monomial basis. 
\if{
Let us note $\Ic_j^k = \{\gb\in\mathbb{N}_k^n~:~\gamma_i=0\textnormal{ if } i \notin I_j \}$. Then the $\lambda_{\ab_j,\bb_j}$ fulfill the following equalities:
\begin{equation}
\label{sparseequal_eq}
     \begin{aligned}
        	  (f(\xb))_{\gb} = & \sum_{j:\gb\in\Ic_j^k}(\phi_j)_{\gb} ~, \quad \forall\gb\in\bigcup_{j=1}^{m}\Ic_j^k~,\\   
      		  (\phi_j)_{\gb} = & \sum_{|\ab_j+\bb_j|\leq k}\lambda_{\ab_j,\bb_j}(h_{\ab_j,\bb_j})_{\gb} \,, \quad j=1,\dots,m .
     \end{aligned}
\end{equation}
}\fi
We also obtain a hierarchy of LP relaxations to approximate the solution of polynomial optimization problems. For the sake of conciseness, we only provide these relaxations as well as their computational costs in the particular context of roundoff error bounds in Section~\ref{contributions_handelman}.

%\footnote{We note that $ (\phi_r)_{\gb}$ can be seen as temporary variables.}

\section{Two new methods to compute roundoff errors bounds}
\label{contributions}
This section is dedicated to our main contributions. We provide two new methods to compute absolute roundoff error bounds using either Bernstein expansions or sparse Krivine-Stengle representations.
Here we consider a given program which implements a polynomial expression $f$ with input variables $\xb$ satisfying a set of input constraints encoded by $\mathbf{X}$. We restrict ourselves to the case where $\mathbf{X}$ is the unit box $[0,1]^n$. \newline
Following the simple rounding model described in Section~\ref{preliminaries_floating_point}, we note $\hat{f}(\xb,\eb)$ the rounded expression of $f$ after introduction of the rounding variables $\eb$ (one additional variable is introduced for each real variable $x_i$ or constant as well as for each arithmetic operation $+$,$\times$ or $-$). For a given machine epsilon $\varepsilon$, these error variables also satisfy a set of constraints encoded by the box $[-\varepsilon, \varepsilon]^m$.
As explained in~\cite[Section 3.1]{real2float}, we can decompose the  roundoff error as follows: $r(\xb, \eb) := \hat{f}(\xb,\eb) - f(\xb) = l(\xb,\eb) + h(\xb, \eb)$, where $l(\xb,\eb) := \sum_{j=1}^m \frac{\partial r(\xb,\eb)} {\partial e_j} (\xb,0)  e_j  = \sum_{j=1}^m s_j(\xb) e_j$. 
One obtains an enclosure of $h$ using interval arithmetic to bound second-order error terms in the Taylor expansion of $r$ w.r.t.~$\eb$ (as in~\cite{fptaylor,real2float}).\newline 
We note $d$ the degree of $l$. %, The maximal value of the linear part of the absolute roundoff error problem 
After dividing each error variable $e_j$ by $\varepsilon$, we now consider the optimization of the (scaled) linear part $l' := l/ \varepsilon$ of the roundoff error. In other words, we focus on computing upper bounds of the maximal absolute value $l'^* := \max_{(\xb, \eb) \in\mathbf{X} \times \mathbf{E} } |l'(\xb,\eb)|$ where $\mathbf{E} = [-1,1]^m$.
%as it allows to describe all the sets similar by a linear transformation. 
%Moreover, as in Section \ref{overview}, we take $\eb\in\mathbf{E} =[-1,1]^m$ as the optimal values are $\varepsilon$ proportional to the optimal values on $[-\varepsilon,\varepsilon]^m$ . 
\subsection{Bernstein expansions of roundoff errors}
\label{contributions_bernstein}
The first method is the approximation of $l'^*$ with the Bernstein expansions. Let $\db$ be the  multi-degree of $l'_{\eb}$. From the above definition of $l$, note that $\db$ is also the multi-degree of $f$. 
For each $\kb \geq \db$, let us note $\overline{l'_{\kb}} := \max_{\ab \leq \kb}\sum_{j = 1}^m |b_{\ab}^{(s_j)}|$ and $\underline{l'_{\kb}} := - \overline{l'_{\kb}}$.
Our procedure is based on the following lemma:
%
%One first writes $l(\xb,\eb)$ as: $l(\xb,\eb) = \sum\limits_{r=1}^m s_r(\xb)e_r$, where $s_1(\xb),\dots,s_n(\xb)$ are polynomials of degree at most $\mathbf{d}$. Then, the following lemma apply:
%
\begin{lemma}%[Bernstein expansion of roundoff errors]
\label{bernth}
For each $\kb \geq \db$, the polynomial $l'(\xb,\eb)$ can be bounded as follows:
%over $\mathbf{X} \times \mathbf{E}$ is given by:
\begin{equation}
\underline{l'_{\kb}}  \leq l'(\xb,\eb) \leq \overline{l'_{\kb}} \,,  \quad \forall (\xb,\eb) \in \mathbf{X} \times \mathbf{E} \,.
\end{equation} 
\end{lemma}
\begin{proof}
We write $l'_{\eb}\in\mathbb{R}[\xb]$ the polynomial $l'(\xb,\eb)$ for a given $\eb\in \mathbf{E}$. Property~\ref{enclosure_prop} provides the enclosure of $l'_{\eb}(\xb)$ w.r.t. $\xb$ for a given $\eb\in \mathbf{E}$:
%Property~\ref{enclosure_prop} provides the enclosure of $l'(\xb,\eb)$ w.r.t. $\xb$ for a given $\eb\in \mathbf{E}$:
%\begin{equation}
%\min_{\ab} b_{\ab}^{(l')}(\eb) \leq l'(\xb,\eb)\leq \max_{\ab} b_{\ab}^{(l')}(\eb) \,, \quad \forall \xb\in[0,1]^n \,.
%\end{equation} 
\begin{equation}
\min_{\ab \leq \kb} b_{\ab}^{(l'_{\eb})} \leq l'_{\eb}(\xb)\leq \max_{\ab \leq \kb} b_{\ab}^{(l'_{\eb})} \,, \quad \forall \xb\in[0,1]^n \,,
\end{equation} 
where each Bernstein coefficient satisfies $b_{\ab}^{(l'_{\eb})} = \sum_{j=1}^m e_j b_{\ab}^{(s_j)}$ by Property \ref{lin_prop} (each $e_j$ being a scalar in $[-1,1]$). 
The proof of  the left inequality comes from:
\begin{align*}
\min_{\eb\in[-1,1]^m}\bigl(\min_{\ab \leq \kb}(\sum_{j=1}^m e_j b_{\ab}^{(s_j)})\bigr) &
= \min_{\ab \leq \kb}\bigl( \min\limits_{\eb\in[-1,1]^m}(\sum\limits_{j=1}^m e_j b_{\ab}^{(s_j)})\bigr)\\
  &= \min\limits_{\ab \leq \kb} \sum\limits_{j=1}^m -|b_{\ab}^{(s_j)}| 
  = - \max\limits_{\ab \leq \kb} \sum\limits_{j=1}^m |b_{\ab}^{(s_j)}|~~.
\end{align*}
%  &= \min\limits_{\ab} -1\cdot\sum\limits_{j=1}^m |b_{\ab}^{(s_j)}|\\
The proof of the right inequality is similar.\hfill\qed
\end{proof}
\begin{remark}
\label{rk:cost1}
The computational cost of $\underline{l'_{\kb}}$ is  $m (\kb + \bone)^{\bone}$  since we need to compute the Bernstein coefficients for each $s_j(\xb)$. This cost is polynomial in the degree and exponential in $n$ but is linear in $m$.
In the implementation described in Section \ref{implementation}, we first compute each $b_{\ab}^{(l'_{\eb})}$ as a function of  $\eb$ and then optimize afterwards over $[-1, 1]^m$.
\end{remark}
\begin{example}
\label{bernex}
For the polynomial $l$ defined in~\eqref{eq:l} (Section~\ref{overview}),  one has $l(x,\eb) = (2x^2-x) e_1 + x^2 e_2 + (x^2-x) e_3$. Applying the above method with $\kb = \db = 2$, one considers the following Bernstein coefficients:
\[ 
b_0^{(l'_{\eb})} = 0, \quad b_1^{(l'_{\eb})} = -\frac{e_1}{2}-\frac{e_3}{2}, \quad b_2^{(l'_{\eb})} = e_1+e_2.
\]
The number of Bernstein coefficients w.r.t.~$x$ is $3$, which is much lower than the one w.r.t.~$(x,\eb)$, which is equal to $24$. 
One can obtain an upper bound (resp.~lower bound)  by taking the maximum (resp.~minimum) of the Bernstein coefficients. 
In this case, $\max_{\eb\in [-1,1]^3} b_1^{(l'_{\eb})} = 0$, $\max_{\eb\in [-1,1]^3} b_2^{(l'_{\eb})} = 1$ and $\max_{\eb\in [-1,1]^3} b_3^{(l'_{\eb})} = 2$. Thus, one obtains $\overline{l'_{\kb}} = 2$ as an upper bound of  $l'^*$  yielding $l^* \leq 2\varepsilon$.
\end{example}
\subsection{Sparse Krivine-Stengle representations of roundoff errors}
\label{contributions_handelman}
Here we explain how to compute lower bounds of $\underline{l'} := \min_{(\xb,\eb) \in \mathbf{X} \times \mathbf{E}} l'(\xb,\eb)$ by using sparse Krivine-Stengle representations. 
We obtain upper bounds of $\overline{l'} := \max_{(\xb,\eb) \in \mathbf{X} \times \mathbf{E}} l'(\xb,\eb)$ in a similar way.

For the sake of consistency with Section~\ref{preliminaries_handelman}, we introduce the variable $\yb \in \mathbb{R}^{n+m}$ defined by $y_j := x_j$, $j=1,\dots,n$ and $y_j := e_{j-n}$, $j=n+1,\dots,n+m$.
Then, one can write the set $\cSet = \mathbf{X}\times\mathbf{E}$  as follows:
 \begin{equation}
 \cSet = \{\yb\in\mathbb{R}^{n+m} : 0 \leq g_j(\yb) \leq 1 \,, \quad j=1,\dots, n+m \} \,,
  \end{equation}
with  $g_j(\yb) := x_j$, for each $j=1,\dots,n$ and $g_j(\yb) := \frac{1}{2}+\frac{e_j}{2}$, for each $j=n+1,\dots,n+m$. 
\begin{lemma}
\label{pattern}
For each $j=1, \dots, m$, let us define $I_j := \{1,\dots,n,n+j\}$ and $J_j := I_j$. Then the sets $I_j$ and $J_j$ satisfy the four conditions stated in~Definition~\ref{sparse_def}.
\end{lemma}
\begin{proof}
The first condition holds as $l'(\yb) = l'(\xb,\eb) = \sum_{j=1}^m s_j(\xb, \eb) e_j =  \sum_{j=1}^m s_j(\yb)  e_j$, with $s_j(\yb)\in\mathbb{R}[\yb,I_j]$. The second and third condition are  obvious. The running intersection property  comes from $I_{j+1} \cap I_j = \{1, \dots, n \} \subseteq I_j$. \hfill \qed
\end{proof}
Given $\ab,\bb \in \mathbb{N}^{n+1}$, one can write $\ab = (\ab',\gamma)$ and $\bb = (\bb',\delta)$, for $\ab',\bb' \in \mathbb{N}^{n}$, $\gamma,\delta \in \mathbb{N}$.
In our case, this gives the following formulation for the polynomial $h_{\ab_j,\bb_j}(\yb)= 
\mathbf{g}^{\ab_j}
( \bone - \mathbf{g})^{\bb_j}$:
\[
h_{\ab_j,\bb_j}(\yb)= h_{\ab_j',\bb_j',\gamma_j,\delta_j}(\xb,\eb)= 
\xb^{\ab_j'}
( \bone - \xb)^{\bb_j'}
( \frac{1}{2}+\frac{e_j}{2})^{\gamma_j}
( \frac{1}{2}-\frac{e_j}{2})^{\delta_j}
 \,. \]
For instance, with the polynomial $l'$ considered in Section~\ref{overview} and depending on $x, e_1, e_2, e_3$, one can consider the multi-indices $\ab_1 = (1,2)$, $\bb_1 = (2,3)$ associated to the roundoff variable $e_1$. Then  $h_{\ab_1,\bb_1}(\yb) = x (1 - x)^2 (\frac{1}{2}+\frac{e_1}{2})^2 ( \frac{1}{2}-\frac{e_1}{2})^{3}$.
 
Now, we consider the following hierarchy of LP relaxations, for each $k \geq d$:
\begin{equation}
\label{sparsetheq}
\begin{aligned}
\underline{l'_k} := \max_{t,\lambda_{\ab_j,\bb_j}} \quad & t  \,,  \\
\text{s.t } \quad & l'-t  = \sum_{j=1}^m\phi_j \,,\\   
\quad & \phi_j  =  \sum_{|\ab_j+\bb_j|\leq k}\lambda_{\ab_j,\bb_j} h_{\ab_j,\bb_j} \,, \quad j=1,\dots,m \,, \\
\quad & \lambda_{\ab_j,\bb_j}\geq 0  \,, \quad j=1,\dots,m \,.
\end{aligned}
\end{equation}
Similarly, we obtain $\overline{l'_k}$ while replacing $\max$ by $\min$ and $l'-t$ by $t - l'$ in LP~\eqref{sparsetheq}.
%Then, it follows from Theorem~\ref{spkrivrep_th}:
\begin{lemma}%[Sparse K-S representations of roundoff errors]
\label{sparseLPrelax_th}
The sequence of optimal values $(\underline{l'_k})$ (resp.~$(\overline{l'_k})$) satisfies $\underline{l'_k} \uparrow \underline{l'}$ (resp.~$\overline{l'_k} \downarrow \overline{l'}$) as $k \rightarrow +\infty$. In addition, $l'_k :=  \max \{|\underline{l'_k}|, |\overline{l'_k}| \} \rightarrow l'^*$ as $k \rightarrow +\infty$.
%In addition, each $\underline{l'_k}$  is a lower bound of $\underline{l'}$ and each $\overline{l'_k}$  is an upper bound of $\overline{l'}$. 
\end{lemma}
\begin{proof}
By construction $(\underline{l'_k})$ is monotone nondecreasing. 
For a given arbitrary $\varepsilon' > 0$, the polynomial $l'- \underline{l'} + \varepsilon'$ is positive over $\cSet$. By Lemma~\ref{pattern}, the subsets $I_j$ and $J_j$ satisfy the four conditions stated in~Definition~\ref{sparse_def}, so we can apply Theorem \ref{spkrivrep_th} to $l'- \underline{l'} + \varepsilon'$. This yields the existence of $\phi_j$, $j=1,\dots,m$, such that $l'- \underline{l'} + \varepsilon' = \sum_{j=1}^m\phi_j$ and $\phi_j = \sum_{|\ab_j+\bb_j|\leq k} \lambda_{\ab_j,\bb_j} h_{\ab_j,\bb_j}$, $j=1,\dots,m$.
Hence, $(\underline{l'} - \varepsilon',\phi_j,\lambda_{\ab_j,\bb_j})$ is feasible for LP~\eqref{sparsetheq}. 
It follows that there exists $k$ such that $\underline{l'_k} \geq l'-\varepsilon'$.
Since $\underline{l'_k} \leq l'$, and $\varepsilon'$ has been arbitrary chosen, we obtain the convergence result for the sequence $(\underline{l'_k})$. The proof is analogous for $(\overline{l'_k})$ and yields $\max \{|\underline{l'_k}|, |\overline{l'_k}| \} \rightarrow \max \{|\underline{l'}|, |\overline{l'}| \} = l'^*$ as $k \to +\infty$, the desired result.\hfill\qed
\end{proof}
%
%As seen in Section \ref{preliminaries_handelman}, Equation \ref{sparsetheq} is an LP problem, as we can equalize, in the monomial basis, the coefficients between the right and the left sides of the equalities we leads to Equation \ref{sparseequal_eq}. 
\begin{remark}
\label{rk:cost2}
In the special case of roundoff error computation, one can prove that the number of variables of LP~\eqref{sparsetheq} is $m \binom {2(n+1)+k} {k} +1$ with a number of constraints equal to $[\frac{mk}{n+1}+1]\binom{n+k}{k}$. This is in contrast with the dense case where the number of LP variables is $\binom {2(n+m)+k} {k} +1$ with a number of  constraints equal to $\binom {n+m+k}{k}$ .
\end{remark}
\renewcommand*{\proofname}{Proof of Remark 2}
\begin{proof}
As we replace a function $\phi$ of dimension $(n+m)$ by a sum of $m$ functions $\phi_j$ of dimension $(n+1)$, the number of coefficients $\lambda_{\ab_j,\bb_j}$ is $m \binom {2(n+1)+k} {k}$. This leads to a total of $m \binom {2(n+1)+k} {k} +1$ variables when adding $t$.

The number of equality constraints is the number of monomials involved in  $\sum_{j=1}^m\phi_j$. Each $\phi_j$ has $\binom {(n+1)+k} {k}$ monomials. However there are redundant monomials between all the $\phi_j$: the ones depending of only $\xb$, and not $\eb$. These $\binom {n+k} {k}$ monomials should appear only once. This leads to a final number of $m \binom {(n+1)+k} {k} - (m-1)\binom {n+k} {k}$ monomials which is equal to $[\frac{mk}{n+1}+1]\binom{n+k}{k}$.\hfill\qed
\end{proof}

%\footnote{This LP size is given with the $(\phi)_{\gb}$ taken as temporary variables, as we can do the substitution.}.
%
\begin{example}
Continuing Example~\ref{bernex}, for the polynomial $l$ defined in~\eqref{eq:l} (Section~\ref{overview}), we consider LP~\eqref{sparsetheq} at the relaxation order $k = d = 3$. This problem involves $3\binom{2\times(1+1)+3}{3} +1 = 106$ variables and $[\frac{3 \times 3}{2}+1]\binom{4}{3} = 22$ constraints. This is in contrast with a dense Krivine-Stengle representation, where the corresponding LP involves $35$ linear equalities and $166$ variables. Computing the values of $\underline{l'_k}$ and $\overline{l'_k}$ provides  an upper bound of $2$ for $l'^*$, yielding $l^* \leq 2\varepsilon$.
\end{example}
\section{Implementation \& Results}
\label{implementation}
\paragraph{\textbf{The  \texttt{FPBern} and \texttt{FPKriSten} software packages.}}
 We provide two distinct software packages to compute certified error bounds of roundoff errors  for programs implementing  polynomial functions with floating point precision. The first tool \href{https://github.com/roccaa/FPBern}{\texttt{FPBern}} relies on the method from Section \ref{contributions_bernstein} and the second tool \href{https://github.com/roccaa/FPKriSten}{\texttt{FPKriSten}} on the method from Section \ref{contributions_handelman}.\newline
\texttt{FPBern} is built on top of the software presented in \cite{dreossiHSCC} to manipulate Bernstein expansions, which includes  a \texttt{C++} module \texttt{FPBern(a)} and  a \texttt{Matlab} module  \texttt{FPBern(b)}. Their main difference is that Bernstein coefficients are computed with double precision floating point arithmetic in \texttt{FPBern(a)} and with rational arithmetic in \texttt{FPBern(b)}. Polynomial operations are handled with  {\sc GiNaC}~\cite{ginac} in \texttt{FPBern(a)} and with \texttt{Matlab Symbolic Toolbox} in \texttt{FPBern(b)}. 
Note that the Bernstein coefficient computations are not fully certified with \texttt{FPBern(a)} yet. We plan to obtain verified upper bounds by using the framework in~\cite{Garloff08}.
\newline
% Passe Victor The first one computes the Bernstein coefficients with a \texttt{C++} double precision implementation. It uses two libraries: \texttt{GiNaC}~\cite{ginac} for symbolic computation, and \texttt{GLPK}.
%\footnote{Acually, the \texttt{GLPK} library is not used for our current problem.}. 
%It is possible to extend, in the future, to rational arithmetic using external library such as \texttt{GMP} (\texttt{GNU Muliple Precision}). \\
% Passe Victor The second module provides certified error bounds using a rational implementation of Bernstein coefficients with the \texttt{Matlab} \texttt{Symbolic Toolbox}. This implementation is also based on the previous work of \cite{dreossiHSCC}.\\
%
\texttt{FPKriSten} is built on top of the {\sc SBSOS} software related to~\cite{sbsos} which handles sparse polynomial optimization problems by solving a hierarchy of convex relaxations. This hierarchy is obtained by mixing Krivine-Stengle  and Putinar representations of positive polynomials.  
To improve the overall performance in our particular case, we only consider the former representation yielding the hierarchy of LP relaxations~\eqref{sparsetheq}. Among several LP solvers,  {\sc Cplex}~\cite{cplex} yields the best performance in our case (see also~\cite{lp_compare} for more comparisons). Polynomials are handled with the {\sc Yalmip}~\cite{YALMIP} toolbox available for \texttt{Matlab}. Even though the semantics of programs considered in this paper is actually much simpler than that considered by other tools such as \texttt{Rosa}~\cite{rosa} or {\sc Fluctuat}~\cite{fluctuat}, we emphasize that those tools may be combined with external non-linear solvers to solve specific sub-problems, a task that either \texttt{FPBern} or \texttt{FPKriSten} can fulfill.
% Passe Victor \texttt{FPKriSten} implementation computes the roundoff error bounds using a sparse Krivine-Stengle hierarchy based on the previous SBSOS \texttt{Matlab} implementation of \cite{sbsos}. Modification of the SBSOS toolbox associated to \cite{sbsos} has been made to handle the construction of a LP only representation, and to improve the overall performance in our particular situation. Moreover, we chose to use the solver {\sc Cplex}~\cite{cplex} to solve the LP problem (see Section \ref{contributions_handelman}). \texttt{FPKriSten} also needs the {\sc Yalmip}~\cite{YALMIP} toolbox available for \texttt{Matlab} to parse polynomials functions.\\
%The program semantic currently handled by both tools is programs implementing polynomial functions with inputs laying inside a box. 
%
\paragraph{\textbf{Experimental results.}}
We tested our two software packages with $20$ programs (see Appendix~\ref{appendix}) where $12$ are existing benchmarks coming from biology, space control and optimization fields, and $8$ are generated as follows, with $\xb = (x_1,\dots,x_n) \in [-1,1]^n$.
\begin{equation}
	\texttt{ex-n-nSum-deg}(\xb) := \sum_{j=0}^{\texttt{nSum}}(\prod_{k=1}^{\texttt{deg}}(\sum_{i=1}^{\texttt{n}}x_i)) \,.
\end{equation}
The first $9$  programs are used  for similar comparison in~\cite[Section 4.1]{real2float}, the following $3$ come from~\cite{SolovyevH13}. Eventually the 8 generated benchmarks allow to evaluate \emph{independently} the performance of the tools w.r.t.~either the number of input variables (through the variable \texttt{n}), the degree (through \texttt{deg}) or the number of error variables (through \texttt{nSum}). 
%without changing the degree or the dimension of the tested polynomial. 
Taking $\xb \in [-1,1]^n$ allows avoiding monotonicity of the polynomial (which could be exploited by the Bernstein techniques).\\
We recall that each program implements a polynomial function $f(\xb)$ with box constrained  input variables. To provide an upper bound of the absolute roundoff error $| f(\xb)-\hat{f}(\xb,\eb) | = | l(\xb,\eb) + h(\xb,\eb) |$, we rely on~\texttt{Real2Float} to generate  $l$ and to bound $h$ (see~\cite[Section 3.1]{real2float}). Then the optimization methods of Section~\ref{contributions} are applied to bound a function $l'$, obtained after linear transformation of $l$, over the unit box. \newline
At a given multi-degree $\kb$, \texttt{FPBern} computes the  bound $\overline{l'_{\kb}}$ (see Lemma~\ref{bernth}). Similarly, at a given relaxation order $k$, \texttt{FPKriSten} computes the bound  $l'_k$ (see Lemma~\ref{sparseLPrelax_th}). 
To achieve fast computations,  the default value of $\kb$ is the multi-degree $\db$ of $l'_{\eb}$ (equal to the multi-degree of the input polynomial $f$) and the default value of $k$ is the degree $d$ of $l'$ (equal to the successor of the degree of $f$).\newline
The experiments were carried out on an Intel Core i7-5600U (2.60Ghz, 16GB) with Ubuntu 14.04LTS, \texttt{Matlab 2015a}, {\sc GiNaC}~1.7.1, and {\sc Cplex}~12.63.
Our benchmark settings are similar to~\cite[Section 4]{real2float} as we compare the accuracy and execution times of our two tools with \texttt{Rosa real compiler}~\cite{rosa}(version from May 2014), \texttt{Real2Float}~\cite{real2float}(version from July 2016) and \texttt{FPTaylor}~\cite{fptaylor} (version from May 2016) on programs implemented in double precision while considering input variables as real variables. All these tools use a simple rounding model (see Section~\ref{preliminaries_floating_point}) and have been executed  with their default parameters.\newline
\begin{table*}[!t]
\begin{center}
\small
\caption{Comparison results of upper bounds for absolute roundoff errors. %among tools implementing simple rounding models. For each model, 
The best results are emphasized using \textbf{bold fonts}.\label{table:error}}{

\begin{tabular}{lccc|ccc|ccc}
\hline
{Benchmark} &  $n$ & $m$ &  $d$ & {\texttt{FPBern}(a)}  & {\texttt{FPBern}(b)}  & {\texttt{FPKriSten}}  & {\texttt{Real2Float}} & {\texttt{Rosa}}  & {\texttt{FPTaylor}} \\
\hline  
{\texttt{rigidBody1}}& 3 & 10 & 3
& $5.33\text{e--}13$ & $5.33\text{e--}13$ & $5.33\text{e--}13$ & $5.33\text{e--}13$ & $5.08\text{e--}13$ & $\mathbf{3.87\textbf{e--}13}$   \\
{\texttt{rigidBody2}}& 3 & 15 & 5
& $6.48\text{e--}11$ & $6.48\text{e--}11$ & $6.48\text{e--}11$ & $6.48\text{e--}11$ & $6.48\text{e--}11$ & $\mathbf{5.24\textbf{e--}11}$  \\
{\texttt{kepler0}}& 6 & 21 & 3 
& ${1.08\text{e--}13}$ & ${1.08\text{e--}13}$ & ${1.08\text{e--}13}$ & ${1.18\text{e--}13}$ & $1.16\text{e--}13$ & $\mathbf{1.05\textbf{e--}13}$   \\
{\texttt{kepler1}}& 4 & 28 & 4  
& $4.23\text{e--13}$ & $\mathbf{4.04\textbf{e--}13}$ & $4.23\text{e--13}$ & $4.47\textbf{e--}13$ & $6.49\text{e--}13$ & ${4.49\text{e--}13}$    \\
{\texttt{kepler2}}& 6 & 42 & 4
& $\mathbf{2.03\textbf{e--}12}$ & $\mathbf{2.03\textbf{e--}12}$ & $\mathbf{2.03\textbf{e--}12}$ & $2.09\textbf{e--}12$ & $2.89\text{e--}12$ & ${2.10\text{e--}12}$   \\
{\texttt{sineTaylor}} &  1 & 13 & 8
& $5.51\text{e--16}$ & $\mathbf{5.48\textbf{e--}16}$ & $5.51\text{e--16}$ & $6.03\textbf{e--}16$ & $9.56\text{e--}16$ & $6.75\text{e--}16$  \\
{\texttt{sineOrder3}}& 1 & 6 & 4
& $1.35\text{e--}15$ & $1.35\text{e--}15$ & $1.25\text{e--}15$ & $1.19\text{e--}15$ & $1.11\text{e--}15$ & $\mathbf{9.97\textbf{e--}16}$   \\
{\texttt{sqroot}} & 1 & 15 & 5
& $1.29\text{e--}15$ & $1.29\text{e--}15$ & $1.29\text{e--}15$ & $1.29\text{e--}15$ & $8.41\text{e--}16$ & $\mathbf{7.13\textbf{e--}16}$  \\
{\texttt{himmilbeau}}& 2 & 11 & 5
& ${2.00\text{e--}12}$ & ${2.00\text{e--}12}$ & ${1.97\text{e--}12}$ & ${1.43\text{e--}12}$ & ${1.43\text{e--}12}$ & $\mathbf{1.32\textbf{e--}12}$    \\
{\texttt{schwefel}}& 3 & 15 & 5 
& ${1.48\text{e--}11}$ & ${1.48\text{e--}11}$ & ${1.48\text{e--}11}$ & ${1.49\text{e--}11}$ & ${1.49\text{e--}11}$ & $\mathbf{1.03\textbf{e--}11}$    \\
{\texttt{magnetism}} & 7 & 27 & 3 
& ${1.27\text{e--}14}$ & ${1.27\text{e--}14}$ & ${1.27\text{e--}14}$ & ${1.27\text{e--}14}$ & ${1.27\text{e--}14}$ & $\mathbf{7.61\textbf{e--}15}$    \\
{\texttt{caprasse}}& 4 & 34 & 5
& ${4.49\text{e--}15}$ & ${4.49\text{e--}15}$ & ${4.49\text{e--}15}$ & ${5.63\text{e--}15}$ & ${5.96\text{e--}15}$ & $\mathbf{3.04\textbf{e--}15}$    \\
\hline
\hline
{\texttt{ex-2-2-5}}& 2 & 9 & 3
& ${2.23\text{e--}14}$ & ${2.23\text{e--}14}$ & ${2.23\text{e--}14}$  & ${2.23\text{e--}14}$ & ${2.23\text{e--}14}$ & $\mathbf{1.96\textbf{e--}14}$   \\
{\texttt{ex-2-2-10}}& 2 & 14 & 3
& ${5.33\text{e--}14}$ & ${5.33\text{e--}14}$ & ${5.33\text{e--}14}$ & ${5.33\text{e--}15}$ & ${5.33\text{e--}14}$ & $\mathbf{4.85\textbf{e--}14}$  \\
{\texttt{ex-2-2-15}}& 2 & 19 & 3
& ${9.55\text{e--}14}$ & ${9.55\text{e--}14}$ & ${9.55\text{e--}14}$ & ${9.55\text{e--}14}$ & ${9.55\text{e--}14}$ & $\mathbf{8.84\textbf{e--}14}$   \\
{\texttt{ex-2-2-20}}& 2 & 24 & 3
& ${1.49\text{e--}13}$ & ${1.49\text{e--}13}$ & ${1.49\text{e--}13}$ & $\texttt{TIMEOUT}$ &  ${1.49\text{e--}13}$ & $\mathbf{1.40\textbf{e--}13}$    \\
{\texttt{ex-2-5-2}}& 2 & 9 & 6
& ${1.67\text{e--}13}$ & ${1.67\text{e--}13}$ & ${1.67\text{e--}13}$ & ${1.67\text{e--}13}$ & ${1.67\text{e--}13}$ & $\mathbf{1.41\textbf{e--}13}$    \\
{\texttt{ex-2-10-2}}& 2 & 14 & 11
& ${1.05\text{e--}11}$ & ${1.05\text{e--}11}$ & ${1.34\text{e--}11}$ & ${1.05\text{e--}11}$ & ${1.05\text{e--}11}$ & $\mathbf{8.76\textbf{e--}12}$    \\
{\texttt{ex-5-2-2}}& 5 & 12 & 3
& ${8.55\text{e--}14}$ & ${8.55\text{e--}14}$ & ${8.55\text{e--}14}$ & ${8.55\text{e--}14}$ & ${8.55\text{e--}14}$ & $\mathbf{7.72\textbf{e--}14}$    \\
{\texttt{ex-10-2-2}}& 10 & 22 & 3
& ${5.16\text{e--}13}$ & $\texttt{TIMEOUT}$ & ${5.16\text{e--}13}$ & ${5.16\text{e--}13}$ & ${5.16\text{e--}13}$ & $\mathbf{4.82\textbf{e--}13}$    \\
\hline
\end{tabular}
}
\end{center}
\end{table*}

Table \ref{table:error}  shows the result of the absolute roundoff error while Table~\ref{table:cpu} displays execution times obtained through averaging over $5$ runs. For each benchmark, we indicate the number $n$ (resp.~$m$) of input (resp.~error) variables as well as the degree $d$ of $l'$. For \texttt{FPKriSten} the {\sc Cplex} solving time in Table~\ref{table:cpu} is given between parentheses. Note that the overall efficiency of the tool could be improved by constructing the hierarchy of LP~\eqref{sparsetheq} with a \texttt{C++} implementation.

Our two methods yield more accurate bounds for the $3$ benchmarks \texttt{kepler1}, \texttt{sineTaylor} and \texttt{kepler2}, which is  the program involving the largest number of error variables.\\
For \texttt{kepler1}, \texttt{FPBern(a)} and \texttt{FPKriSten} are less precise than \texttt{FPBern(b)} but are still $6\%$ more precise than \texttt{Real2Float} and \texttt{FPTaylor} and  $53\%$ more precise than \texttt{Rosa}.
For \texttt{kepler2}, our two tools are $3\%$ (resp. $42\%$) more precise than \texttt{FPTaylor} and \texttt{Real2Float} (resp. \texttt{Rosa}). In addition, Property~\ref{sharp_prop} holds for these three programs with \texttt{FPBern(b)}, which ensures bound optimality. For all other benchmarks \texttt{FPTaylor} provides the most accurate upper bounds.  Our tools are more accurate than \texttt{Real2Float} except for \texttt{sineOrder3} and \texttt{himmilbeau}. In particular, for \texttt{himmilbeau}, \texttt{FPBern} and \texttt{FPKriSten} are $40\%$ (resp.~$50\%$) less precise than  \texttt{Real2Float} (resp.~\texttt{FPTaylor}). One way to obtain better bounds would be to increase the degree $\kb$ (resp.~ relaxation order $k$) within \texttt{FPBern} (resp.~\texttt{FPKriSten}). Preliminary experiments indicate modest accuracy improvement at the expense of performance.

\begin{table}[!t]
\begin{center}
\small
\caption{Comparison of execution times (in seconds) for absolute roundoff error bounds. 
%among tools implementing simple rounding models. 
%For all model, the winner results are for \texttt{FPBern(a)} emphasized by \textit{italic fonts}.
For each model, the best results are emphasized using \textbf{bold fonts}.
\label{table:cpu}}{
%\begin{tabular}{p{2.3cm}c|ccc|ccc}
%\begin{tabular}{lc|ccc|ccc}
%\hline
%{Benchmark} &  {$(n,m,k)$} & {\texttt{FPBern}(a)}  & {\texttt{FPBern}(b)}  & {\texttt{FPKriSten}}  & {\texttt{Real2Float}} & {\texttt{Rosa}}  & {\texttt{FPTaylor}} \\
%\hline
%{\texttt{rigidBody1}}&$(1,10,3)$ & $\mathbf{0.003}$
%& $0.88$ & $0.22(0.02)$ & $0.58$ & $0.13$ & $1.84$  \\
%{\texttt{rigidBody2}}& $(3,15,5)$ & $\mathbf{0.008}$
%& $1.87$ & $2.78(0.47)$ &$0.26$ & $2.17$ & $3.01$  \\
%{\texttt{kepler0}}&$(6,21,3)$ &  $\mathbf{0.06}$
%& $9.62$ & $1.93(0.18)$ & $0.22$ & $3.78$ & $4.93$ \\
%{\texttt{kepler1}}& $(4,28,4)$ & $\mathbf{0.10}$
%& $6.91$ & $3.93(0.53)$ & $17.6$ & $63.1$ & $9.33$  \\
%{\texttt{kepler2}}&$(6,42,4)$ & $\mathbf{2.06}$
%& $64.90$ & $20.5(3.75)$ & $16.5$ & $106$ & $19.1$  \\
%{\texttt{sineTaylor}}&$(1,13,8)$ & $\mathbf{0.003}$
%& $0.50$ & $0.92(0.27)$ & $1.05$ & $3.50$ & $2.91$  \\
%{\texttt{sineOrder3}}&$(1,6,4)$ & $\mathbf{0.0005}$
%& $0.27$ & $0.08(0.006)$ & $0.40$ & $0.48$ & $1.90$ \\
%{\texttt{sqroot}} &$(1,15,5)$ & $\mathbf{0.002}$
%& $0.34$ & $0.24(0.02)$ & $0.14$ & $0.77$ & $2.70$  \\
%{\texttt{himmilbeau}}&$(2,11,5)$ & $\mathbf{0.01}$
%& $1.72$ & $0.77(0.22)$ & $0.20$ & $2.51$ & $3.28$ \\
%\hline
%\end{tabular}

\begin{tabular}{lccc|ccc|ccc}
\hline
{Benchmark} &  $n$ & $m$ & $d$ & {\texttt{FPBern}(a)}  & {\texttt{FPBern}(b)}  & {\texttt{FPKriSten}}  & {\texttt{Real2Float}} & {\texttt{Rosa}}  & {\texttt{FPTaylor}} \\
\hline
{\texttt{rigidBody1}}& 3 & 10 & 3  &  $\mathbf{5\textbf{e--}4}$
& $0.88$ & $0.22(0.02)$ & $0.58$ & ${0.13}$ & $1.84$  \\
{\texttt{rigidBody2}}& 3 & 15 & 5 & $\mathbf{2\textbf{e--}3}$
& $1.87$ & $2.78(0.47)$ &${0.26}$ & $2.17$ & $3.01$  \\
{\texttt{kepler0}}& 6 & 21 & 3 & $\mathbf{4\textbf{e--}3}$
& $9.62$ & $1.93(0.18)$ & ${0.22}$ & $3.78$ & $4.93$ \\
{\texttt{kepler1}}& 4 & 28 & 4  & $\mathbf{6\textbf{e--}3}$
& $6.91$ & ${3.93(0.53)}$ & $17.6$ & $63.1$ & $9.33$  \\
{\texttt{kepler2}}& 6 & 42 & 4  & $\mathbf{5\textbf{e--}2}$
& $64.9$ & $20.5(3.75)$ & ${16.5}$ & $106$ & $19.1$  \\
{\texttt{sineTaylor}} &  1 & 13 & 8  & $\mathbf{6\textbf{e--}4}$
& ${0.50}$ & $0.92(0.27)$ & $1.05$ & $3.50$ & $2.91$  \\
{\texttt{sineOrder3}}& 1 & 6 & 4  & $\mathbf{2\textbf{e--}4}$
& $0.27$ & ${0.08(0.01)}$ & $0.40$ & $0.48$ & $1.90$ \\
{\texttt{sqroot}} & 1 & 15 & 5  & $\mathbf{2\textbf{e--}4}$
& $0.34$ & $0.24(0.02)$ & ${0.14}$ & $0.77$ & $2.70$  \\
{\texttt{himmilbeau}}& 2 & 11 & 5  & $\mathbf{1\textbf{e--}3}$
& $1.72$ & $0.77(0.22)$ & ${0.20}$ & $2.51$ & $3.28$ \\
{\texttt{schwefel}}& 3 & 15 & 5  & $\mathbf{2\textbf{e--}3}$
& $3.04$ & ${2.90(0.56)}$ & ${0.23}$ & $3.91$ & $0.53$ \\
{\texttt{magnetism}} & 7 & 27 & 3  & $\mathbf{9\textbf{e--}2}$
& $176$ & $3.07(0.26)$ & ${0.29}$ & $1.95$ & $5.91$  \\
{\texttt{caprasse}}& 4 & 34 & 5 & $\mathbf{6\textbf{e--}3}$
& $6.03$ & $18.8(4.89)$ & ${3.63}$ & $17.6$ & $12.2$ \\
\hline
\hline
{\texttt{ex-2-2-5}}& 2 & 9 & 3 &  $\mathbf{4\textbf{e--}4}$
& $0.69$ & $0.12(0.01)$ & $0.07$ & $4.20$ & $2.30$  \\
{\texttt{ex-2-2-10}}& 2 & 14 & 3 &  $\mathbf{5\textbf{e--}4}$
& $0.71$ & $0.17(0.01)$ & $0.35$ & $4.75$ & $3.42$  \\
{\texttt{ex-2-2-15}}& 2 & 19 & 3 &  $\mathbf{6\textbf{e--}4}$
& $0.72$ & $0.23(0.02)$ & $9.75$ & $5.33$ & $4.91$  \\
{\texttt{ex-2-2-20}}& 2 & 24 & 3 &  $\mathbf{8\textbf{e--}4}$
& $0.73$ & $0.28(0.02)$ & $\texttt{TIMEOUT}$ & $6.28$ & $6.27$  \\
{\texttt{ex-2-5-2}}& 2 & 9 & 6 &  $\mathbf{2\textbf{e--}2}$
& $2.34$ & $1.23(0.26)$ & $0.27$ & $4.26$ & $2.53$  \\
{\texttt{ex-2-10-2}}& 2 & 14 & 11 &  $\mathbf{2\textbf{e--}2}$
& $7.34$ & $96.9(58.5)$ & $49.2$ & $9.37$ & $5.07$  \\
{\texttt{ex-5-2-2}}& 5 & 12 & 3 &  $\mathbf{8\textbf{e--}3}$
& $18.3$ & $0.70(0.08)$ & $0.21$ & $4.45$ & $12.3$  \\
{\texttt{ex-10-2-2}}& 10 & 22 & 3 &  ${39.5}$
& $\texttt{TIMEOUT}$ & $6.11(0.6)$ & $30.7$ & $\mathbf{5.34}$ & $34.6$  \\
\hline
\end{tabular}
}
\end{center}
\end{table}
%

%% Ajout du 06/01/2017

\texttt{FPBern(a)} is the fastest for almost all benchmarks  (except program \texttt{ex-10-2-2} where \texttt{Rosa} yields best performance). \texttt{FPBern(b)} is much slower due to its \texttt{Matlab} implementation, and the use of certified rational arithmetic. We plan to implement a similar certification scheme within \texttt{FPBern(a)}.\\
On the first $12$ benchmarks, \texttt{FPBern(a)} is always the fastest while having a similar precision to \texttt{Real2Float} or \texttt{Rosa}. \\%We can further improve the precision by increasing the degree (resp.~relaxation order) of the Bernstein expansion (resp.~Krivine-Stengle representation) yielding better accuracy than \texttt{Rosa} and \texttt{Real2Float}.\\
The results obtained with the $8$ generated benchmarks emphasize the limitations of each method. The Bernstein method performs very well when the number of input variables is low, even if the degree increases, as shown in the results for the 6 programs from \texttt{ex-2-2-5} to \texttt{ex-2-10-2}. This is related to the polynomial dependancy w.r.t.~the degree when fixing the number of input variables.
However, for the last $2$ programs \texttt{ex-5-2-2} and \texttt{ex-10-2-2} where the dimension increases, the computation time increases exponentially. This confirms the theoretical result stated in~Remark~\ref{rk:cost1} as the number of Bernstein coefficients is exponential w.r.t.~the dimension at fixed degree.
%the addressed problem is NP-hard

On the same programs, the method based on Krivine-Stengle representations performs better when the dimension increases, at fixed degree. This confirms the constraint dependency w.r.t.~$[\frac{m k}{n+1}+1] \binom{n+k}{k}$ stated in Remark~\ref{rk:cost2}. 

Results for the $4$ programs from \texttt{ex-2-2-5} to \texttt{ex-2-2-20} also indicate that our methods are the least sensible to an increase of error variables.
We note that \texttt{FPKriSten} is often the second fastest tool.

Let us now provide an overall evaluation of our tools. Our tools are comparable with \texttt{Real2Float} (resp. \texttt{Rosa}) in terms of accuracy and faster than them. In comparison with \texttt{FPTaylor}, our tools are in general less precise but still very competitive in accuracy, and they outperform \texttt{FPTaylor} in computation time. A salient advantage of our tools, in particular \texttt{FPKriSten}, over \texttt{FPTaylor} is a good trade-off between computation time and accuracy for large polynomials. As we can see from the experimental results, for \texttt{ex-10-2-2}, \texttt{FPKriSten} took only $6.11$s while \texttt{FPTaylor} took $34.6$s for comparable precisions. Note that the experimentations were done with \texttt{FPBern(b)} and \texttt{FPKriSten} implemented in \texttt{Matlab};  their  \texttt{C++} implementations would allow a significant speed-up.

The good time performances of our tools come from the exploitation of sparsity. Indeed, a direct Bernstein expansion of the polynomial $l$ associated to \texttt{kepler2} leads to compute  $3^6 \times 2^{42}$ coefficients against $42 \times 3^6$ with \texttt{FPBern}. Similarly, dense Krivine-Stengle representations yield an LP with $\binom{100}{4} + 1 = 3 \ 921 \ 226$ variables while LP~\eqref{sparsetheq} involves $42 \binom{18}{4} + 1 = 128 \ 521$ variables.

\section{Conclusion and Future Works}
\label{conclusion}
We propose two new methods to compute upper bounds of absolute roundoff errors occurring while
executing polynomials programs with floating point precision. The first method uses symbolic Bernstein expansions of polynomials while the second one relies on a hierarchy of LP relaxations derived from sparse Krivine-Stengle representations. The overall computational cost is drastically reduced compared to the dense problem, thanks to a specific exploitation of the sparsity pattern between input and error variables, yielding promising experimental results.
%Numerical experiments obtained with our two software \texttt{FPBern}  and \texttt{FPKriSten} show very promising results as they yield similar performance compared to existing competitive tools while often providing tighter error bounds.

Our approach is currently limited to programs implementing polynomials with box constrained variables. First, a direction of further research investigation is an extension to handle more complicated input sets. Extending to semialgebraic sets is theoretically possible with the hierarchy of LP relaxations based on sparse Krivine-Stengle representations but requires careful implementation in order not to compromise efficiency. For our method based on Bernstein expansions, it would be worth adapting the techniques described in~\cite{Mantzaflaris2010} to obtain polygonal approximations of semialgebraic sets. Second, we intend to aim at formal verification of bounds by interfacing either~\texttt{FPBern} with the PVS libraries~\cite{munoz13} related to Bernstein expansions, or \texttt{FPKirSten} with the $\texttt{Coq}$ libraries available in $\texttt{Real2Float}$~\cite{real2float}. Finally, a delicate but important open problem is to apply such optimization techniques in order to handle roundoff errors of programs implementing finite or infinite loops as well as conditional statements.

%\textbf{\large Acknowledgement}\\
%\newline
%\normalsize
%We gratefully acknowledge the support of Agence Nationale de la 
%Recherche (ANR) through the CADMIDIA project (grant ANR-13-CESA-0008-03).
%\vspace{-0.4cm}
\newpage
\bibliographystyle{acm}
%\bibliography{bibliography}

\begin{thebibliography}{10}

\bibitem{lp_compare}
Decision tree for optimization software.
\newblock \url{http://plato.la.asu.edu/bench.html}.
\newblock Accessed: 2016-10-18.

\bibitem{ginac}
{\sc Bauer, C., Frink, A., and Kreckel, R.}
\newblock Introduction to the ginac framework for symbolic computation within
  the c++ programming language.
\newblock {\em J. Symb. Comput. 33}, 1 (Jan. 2002), 1--12.

\bibitem{Constantinides}
{\sc Boland, D., and Constantinides, G.~A.}
\newblock Automated precision analysis: A polynomial algebraic approach.
\newblock In {\em FCCM'10\/} (2010), pp.~157--164.

\bibitem{CoqProofAssistant}
{The Coq Proof Assistant}, 2016.
\newblock \url{http://coq.inria.fr/}.

\bibitem{rosa}
{\sc Darulova, E., and Kuncak, V.}
\newblock {Towards a Compiler for Reals}.
\newblock Tech. rep., {Ecole Polytechnique Federale de Lausanne}, 2016.

\bibitem{gappa}
{\sc Daumas, M., and Melquiond, G.}
\newblock {Certification of Bounds on Expressions Involving Rounded Operators}.
\newblock {\em ACM Trans. Math. Softw. 37}, 1 (Jan. 2010), 2:1--2:20.

\bibitem{fluctuat}
{\sc Delmas, D., Goubault, E., Putot, S., Souyris, J., Tekkal, K., and
  Védrine, F.}
\newblock Towards an industrial use of fluctuat on safety-critical avionics
  software.
\newblock In {\em Formal Methods for Industrial Critical Systems}, M.~Alpuente,
  B.~Cook, and C.~Joubert, Eds., vol.~5825 of {\em Lecture Notes in Computer
  Science}. Springer Berlin Heidelberg, 2009, pp.~53--69.

\bibitem{dreossiHSCC}
{\sc Dreossi, T., and Dang, T.}
\newblock Parameter synthesis for polynomial biological models.
\newblock In {\em Proceedings of the 17th international conference on Hybrid
  systems: computation and control\/} (2014), ACM, pp.~233--242.

\bibitem{berngarloff}
{\sc Garloff, J.}
\newblock Convergent bounds for the range of multivariate polynomials.
\newblock In {\em Interval Mathematics 1985}. Springer, 1986, pp.~37--56.

\bibitem{Garloff08}
{\sc Garloff, J., and Smith, A.~P.}
\newblock Rigorous affine lower bound functions for multivariate polynomials
  and their use in global optimisation.
\newblock 2008.

\bibitem{schweighofer06}
{\sc Grimm, D., Netzer, T., and Schweighofer, M.}
\newblock A note on the representation of positive polynomials with structured
  sparsity.
\newblock {\em Archiv der Mathematik 89}, 5 (2007), 399--403.

\bibitem{hollight}
{\sc Harrison, J.}
\newblock {HOL Light: A Tutorial Introduction}.
\newblock In {\em FMCAD\/} (1996), M.~K. Srivas and A.~J. Camilleri, Eds.,
  vol.~1166 of {\em Lecture Notes in Computer Science}, Springer, pp.~265--269.

\bibitem{cplex}
{\sc {ILOG, Inc}}.
\newblock {ILOG CPLEX}: High-performance software for mathematical programming
  and optimization, 2006.
\newblock See \url{http://www.ilog.com/products/cplex/}.

\bibitem{krivineanneaux}
{\sc Krivine, J.-L.}
\newblock Anneaux pr{\'e}ordonn{\'e}s.
\newblock {\em Journal d'analyse math{\'e}matique 12}, 1 (1964), 307--326.

\bibitem{Lasserre2009Moments}
{\sc Lasserre, J.}
\newblock {\em {Moments, Positive Polynomials and Their Applications}}.
\newblock Imperial College Press optimization series. Imperial College Press,
  2009.

\bibitem{bsos}
{\sc Lasserre, J.~B., Toh, K.-C., and Yang, S.}
\newblock A bounded degree sos hierarchy for polynomial optimization.
\newblock {\em EURO Journal on Computational Optimization\/} (2015), 1--31.

\bibitem{laurent2009sums}
{\sc Laurent, M.}
\newblock Sums of squares, moment matrices and optimization over polynomials.
\newblock In {\em Emerging applications of algebraic geometry}. Springer, 2009,
  pp.~157--270.

\bibitem{YALMIP}
{\sc Löfberg, J.}
\newblock Yalmip : A toolbox for modeling and optimization in {MATLAB}.
\newblock In {\em Proceedings of the CACSD Conference\/} (Taipei, Taiwan,
  2004).

\bibitem{real2float}
{\sc Magron, V., Constantinides, G., and Donaldson, A.}
\newblock Certified roundoff error bounds using semidefinite programming.
\newblock {\em arXiv preprint arXiv:1507.03331\/} (2015).

\bibitem{Mantzaflaris2010}
{\sc Mantzaflaris, A., and Mourrain, B.}
\newblock {\em A Subdivision Approach to Planar Semi-algebraic Sets}.
\newblock Springer Berlin Heidelberg, Berlin, Heidelberg, 2010, pp.~104--123.

\bibitem{mourrain2009}
{\sc Mourrain, B., and Pavone, J.~P.}
\newblock Subdivision methods for solving polynomial equations.
\newblock {\em Journal of Symbolic Computation 44}, 3 (2009), 292--306.

\bibitem{munoz13}
{\sc Mu{\~{n}}oz, C., and Narkawicz, A.}
\newblock Formalization of a representation of {B}ernstein polynomials and
  applications to global optimization.
\newblock {\em Journal of Automated Reasoning 51}, 2 (August 2013), 151--196.

\bibitem{smithThesis}
{\sc Smith, A.~P.}
\newblock {\em Enclosure methods for systems of polynomial equations and
  inequalities}.
\newblock PhD thesis, 2012.

\bibitem{SolovyevH13}
{\sc Solovyev, A., and Hales, T.~C.}
\newblock Formal verification of nonlinear inequalities with taylor interval
  approximations.
\newblock In {\em {NASA} Formal Methods, 5th International Symposium, {NFM}
  2013, Moffett Field, CA, USA, May 14-16, 2013. Proceedings\/} (2013),
  G.~Brat, N.~Rungta, and A.~Venet, Eds., vol.~7871 of {\em Lecture Notes in
  Computer Science}, Springer, pp.~383--397.

\bibitem{fptaylor}
{\sc Solovyev, A., Jacobsen, C., Rakamari\'c, Z., and Gopalakrishnan, G.}
\newblock {Rigorous Estimation of Floating-Point Round-off Errors with Symbolic
  {Taylor} Expansions}.
\newblock In {\em Proceedings of the 20th International Symposium on Formal
  Methods (FM)\/} (2015), N.~Bj{\o}rner and F.~de~Boer, Eds., vol.~9109 of {\em
  Lecture Notes in Computer Science}, Springer, pp.~532--550.

\bibitem{stengle}
{\sc Stengle, G.}
\newblock A nullstellensatz and a positivstellensatz in semialgebraic geometry.
\newblock {\em Mathematische Annalen 207}, 2 (1974), 87--97.

\bibitem{sbsos}
{\sc Weisser, T., Lasserre, J.-B., and Toh, K.-C.}
\newblock A bounded degree sos hierarchy for large scale polynomial
  optimization with sparsity.
\newblock {\em arXiv preprint arXiv:1607.01151\/} (2016).

\bibitem{IEEE754}
{\sc Zuras, D., Cowlishaw, M., Aiken, A., Applegate, M., Bailey, D., Bass, S.,
  Bhandarkar, D., Bhat, M., Bindel, D., Boldo, S., et~al.}
\newblock Ieee standard for floating-point arithmetic.
\newblock {\em IEEE Std 754-2008\/} (2008), 1--70.

\end{thebibliography}

\newpage
\appendix
\section{Polynomial Program Benchmarks}
\label{appendix}

\begin{itemize} 
	\item rigibody1 : $(x_1,x_2,x_3) \mapsto -x_1x_2-2x_2x_3-x_1-x_3$ defined on $[-15,15]^3$.
	\item rigibody2 : $(x_1,x_2,x_3) \mapsto 2x_1x_2x_3+6x_3^2-x_2^2x_1x_3-x_2$ defined on $[-15,15]^3$.
	\item kepler0 : $(x_1,x_2,x_3,x_4,x_5,x_6) \mapsto x_2x_5+x_3x_6-x_2x_3-x_5x_6+x_1(-x_1+x_2+x_3-x_4+x_5+x_6)$ defined on $[4,6.36]^6$.
	\item kepler1 : $(x_1,x_2,x_3,x_4) \mapsto x_1x_4(-x_1+x_2+x_3-x_4)+x_2(x_1-x_2+x_3+x_4)+x_3(x_1+x_2-x_3+x_4)-x_2x_3x_4-x_1x_3-x_1x_2-x_4$ defined on $[4,6.36]^4$.
	\item kepler2 : $(x_1,x_2,x_3,x_4,x_5,x_6) \mapsto x_1x_4(-x_1+x_2+x_3-x_4+x_5+x_6)+x_2x_5(x_1-x_2+x_3+x_4-x_5+x_6)+x_3x_6(x_1+x_2-x_3+x_4+x_5-x_6)-x_2x_3x_4-x_1x_3x_5-x_1x_2x_6-x_4x_5x_6$ defined on $[4,6.36]^6$.
	\item sineTaylor : $x \mapsto x-\frac{x^3}{6.0}+\frac{x^5}{120.0}-\frac{x^7}{5040.0}$ defined on $[-1.57079632679,1.57079632679]$.
	\item sineOrder3 : $x \mapsto 0.954929658551372 x-0.12900613773279798 x^3$ defined on $[-2,2]$.
	\item sqroot : $x \mapsto 1.0+0.5x-0.125x^2+0.0625x^3-0.0390625x^4$ defined on $[0,1]$.
	\item himmilbeau : $(x_1,x_2) \mapsto (x_1^2+x_2-11)^2+(x_1+x_2^2-7)^2$ defined on $[-5,5]^2$.
	\item schwefel : $(x_1,x_2,x_3) \mapsto (x_1-x_2)^2+(x_2-1)^2+(x_1-x_3^2)^2+(x_3-1)^2$ defined on $[-10,10]^3$.
	\item magnetism : $(x_1,x_2,x_3,x_4,x_5,x_6,x_7) \mapsto x_1^2+2 x_2^2+2 x_3^2+2 x_4^2+2 x_5^2+2 x_6^2+2 x_7^2-x_1$ defined on $[-1,1]^7$.
	\item caprasse : $(x_1,x_2,x_3,x_4) \mapsto x_1 x_3^3+4 x_2 x_3^2 x_4+4 x_1 x_3 x_4^2+2 x_2 x_4^3+4 x_1 x_3+4 x_3^2-10 x_2 x_4-10 x_4^2+2$ defined on $[-0.5,0.5]^4$.
	
\end{itemize} 

\end{document}